\documentclass[11pt]{article}
\usepackage{amsmath,amssymb,amsbsy,amsfonts,amsthm,latexsym,amsopn,amstext,amsxtra,euscript,amscd,url}
\usepackage[paperheight=10in,paperwidth=7.5in,margin=1in]{geometry}

\begin{document}




\newfont{\teneufm}{eufm10}
\newfont{\seveneufm}{eufm7}
\newfont{\fiveeufm}{eufm5}
%
%
\newfam\eufmfam
                    \textfont\eufmfam=\teneufm \scriptfont\eufmfam=\seveneufm
                    \scriptscriptfont\eufmfam=\fiveeufm

%
%
\def\frak#1{{\fam\eufmfam\relax#1}}
%


\def\bbbr{{\rm I\!R}} 
\def\bbbc{{\rm I\!C}} 
\def\bbbm{{\rm I\!M}}
\def\bbbn{{\rm I\!N}} 
\def\bbbf{{\rm I\!F}}
\def\bbbh{{\rm I\!H}}
\def\bbbk{{\rm I\!K}}
\def\bbbl{{\rm I\!L}}
\def\bbbp{{\rm I\!P}}
\newcommand{\lcm}{{\rm lcm}}
\def\bbbone{{\mathchoice {\rm 1\mskip-4mu l} {\rm 1\mskip-4mu l}
{\rm 1\mskip-4.5mu l} {\rm 1\mskip-5mu l}}}
\def\bbbc{{\mathchoice {\setbox0=\hbox{$\displaystyle\rm C$}\hbox{\hbox
to0pt{\kern0.4\wd0\vrule height0.9\ht0\hss}\box0}}
{\setbox0=\hbox{$\textstyle\rm C$}\hbox{\hbox
to0pt{\kern0.4\wd0\vrule height0.9\ht0\hss}\box0}}
{\setbox0=\hbox{$\scriptstyle\rm C$}\hbox{\hbox
to0pt{\kern0.4\wd0\vrule height0.9\ht0\hss}\box0}}
{\setbox0=\hbox{$\scriptscriptstyle\rm C$}\hbox{\hbox
to0pt{\kern0.4\wd0\vrule height0.9\ht0\hss}\box0}}}}
\def\bbbq{{\mathchoice {\setbox0=\hbox{$\displaystyle\rm
Q$}\hbox{\raise
0.15\ht0\hbox to0pt{\kern0.4\wd0\vrule height0.8\ht0\hss}\box0}}
{\setbox0=\hbox{$\textstyle\rm Q$}\hbox{\raise
0.15\ht0\hbox to0pt{\kern0.4\wd0\vrule height0.8\ht0\hss}\box0}}
{\setbox0=\hbox{$\scriptstyle\rm Q$}\hbox{\raise
0.15\ht0\hbox to0pt{\kern0.4\wd0\vrule height0.7\ht0\hss}\box0}}
{\setbox0=\hbox{$\scriptscriptstyle\rm Q$}\hbox{\raise
0.15\ht0\hbox to0pt{\kern0.4\wd0\vrule height0.7\ht0\hss}\box0}}}}
\def\bbbt{{\mathchoice {\setbox0=\hbox{$\displaystyle\rm
T$}\hbox{\hbox to0pt{\kern0.3\wd0\vrule height0.9\ht0\hss}\box0}}
{\setbox0=\hbox{$\textstyle\rm T$}\hbox{\hbox
to0pt{\kern0.3\wd0\vrule height0.9\ht0\hss}\box0}}
{\setbox0=\hbox{$\scriptstyle\rm T$}\hbox{\hbox
to0pt{\kern0.3\wd0\vrule height0.9\ht0\hss}\box0}}
{\setbox0=\hbox{$\scriptscriptstyle\rm T$}\hbox{\hbox
to0pt{\kern0.3\wd0\vrule height0.9\ht0\hss}\box0}}}}
\def\bbbs{{\mathchoice
{\setbox0=\hbox{$\displaystyle     \rm S$}\hbox{\raise0.5\ht0\hbox
to0pt{\kern0.35\wd0\vrule height0.45\ht0\hss}\hbox
to0pt{\kern0.55\wd0\vrule height0.5\ht0\hss}\box0}}
{\setbox0=\hbox{$\textstyle        \rm S$}\hbox{\raise0.5\ht0\hbox
to0pt{\kern0.35\wd0\vrule height0.45\ht0\hss}\hbox
to0pt{\kern0.55\wd0\vrule height0.5\ht0\hss}\box0}}
{\setbox0=\hbox{$\scriptstyle      \rm S$}\hbox{\raise0.5\ht0\hbox
to0pt{\kern0.35\wd0\vrule height0.45\ht0\hss}\raise0.05\ht0\hbox
to0pt{\kern0.5\wd0\vrule height0.45\ht0\hss}\box0}}
{\setbox0=\hbox{$\scriptscriptstyle\rm S$}\hbox{\raise0.5\ht0\hbox
to0pt{\kern0.4\wd0\vrule height0.45\ht0\hss}\raise0.05\ht0\hbox
to0pt{\kern0.55\wd0\vrule height0.45\ht0\hss}\box0}}}}
\def\bbbz{{\mathchoice {\hbox{$\sf\textstyle Z\kern-0.4em Z$}}
{\hbox{$\sf\textstyle Z\kern-0.4em Z$}}
{\hbox{$\sf\scriptstyle Z\kern-0.3em Z$}}
{\hbox{$\sf\scriptscriptstyle Z\kern-0.2em Z$}}}}
\def\ts{\thinspace}

\newtheorem{theorem}{Theorem}[section]
\newtheorem{lemma}[theorem]{Lemma}
\newtheorem{claim}[theorem]{Claim}
\newtheorem{cor}[theorem]{Corollary}
\newtheorem{prop}[theorem]{Proposition}
\newtheorem{definition}{Definition}
\newtheorem{question}[theorem]{Open Question}
\newtheorem{remark}[theorem]{Remark}

\def\squareforqed{\hbox{\rlap{$\sqcap$}$\sqcup$}}
\def\qed{\ifmmode\squareforqed\else{\unskip\nobreak\hfil
\penalty50\hskip1em\null\nobreak\hfil\squareforqed
\parfillskip=0pt\finalhyphendemerits=0\endgraf}\fi}

\def\cA{{\mathcal A}}
\def\cB{{\mathcal B}}
\def\cC{{\mathcal C}}
\def\cD{{\mathcal D}}
\def\cE{{\mathcal E}}
\def\cF{{\mathcal F}}
\def\cG{{\mathcal G}}
\def\cH{{\mathcal H}}
\def\cI{{\mathcal I}}
\def\cJ{{\mathcal J}}
\def\cK{{\mathcal K}}
\def\cL{{\mathcal L}}
\def\cM{{\mathcal M}}
\def\cN{{\mathcal N}}
\def\cO{{\mathcal O}}
\def\cP{{\mathcal P}}
\def\cQ{{\mathcal Q}}
\def\cR{{\mathcal R}}
\def\cS{{\mathcal S}}
\def\cT{{\mathcal T}}
\def\cU{{\mathcal U}}
\def\cV{{\mathcal V}}
\def\cW{{\mathcal W}}
\def\cX{{\mathcal X}}
\def\cY{{\mathcal Y}}
\def\cZ{{\mathcal Z}}

\newcommand{\comm}[1]{\marginpar{%
\vskip-\baselineskip 
\raggedright\footnotesize
\itshape\hrule\smallskip#1\par\smallskip\hrule}}





\def\MOV{{\bf{MOV}}}

\hyphenation{re-pub-lished}

\def\ord{{\mathrm{ord}}}
\def\Nm{{\mathrm{Nm}}}
\renewcommand{\vec}[1]{\mathbf{#1}}

\def \C{{\bbbc}}
\def \F{{\bbbf}}
\def \L{{\bbbl}}
\def \K{{\bbbk}}
\def \Z{{\bbbz}}
\def \N{{\bbbn}}
\def \Q{{\bbbq}}
\def\E{{\mathbf E}}
\def\H{{\mathbf H}}
\def\G{{\mathcal G}}
\def\O{{\mathcal O}}
\def\cS{{\mathcal S}}
\def \R{{\bbbr}}
\def\Fp{\F_p}
\def \fp{\Fp^*}
\def\\{\cr}
\def\({\left(}
\def\){\right)}
\def\fl#1{\left\lfloor#1\right\rfloor}
\def\rf#1{\left\lceil#1\right\rceil}

\def\Zm{\Z_m}
\def\Zt{\Z_t}
\def\Zp{\Z_p}
\def\Um{\cU_m}
\def\Ut{\cU_t}
\def\Up{\cU_p}

\def\ep{{\mathbf{e}}_p}
\def\HH{\cH}

\def \Prob{{\mathrm {}}}

\def\LC{{\cL}_{C,\cF}(Q)}
\def\LCn{{\cL}_{C,\cF}(nG)}
\def\Tr{\mathrm{Tr}\,}

\def\taubar{\overline{\tau}}
\def\sigmabar{\overline{\sigma}}
\def\Fn{\F_{q^n}}
\def\En{\E(\Fn)}

\def\mand{\qquad \mbox{and} \qquad}

\newcommand{\ZZ}{{\mathbb Z}}
\newcommand{\PP}{{\mathbb P}^1}

\newcommand{\FF}{{\mathbb F}}
\newcommand{\QQ}{{\mathbb Q}}
\newcommand{\fq}{{\FF_q}}
\newcommand{\fqstar}{{\FF^*_q}}
\newcommand{\fqm}{{\FF_{q^m}}}
\newcommand{\fqn}{{\FF_{q^n}}}
\newcommand{\fptwo}{{\FF_{p^2}}}
\newcommand{\fqtwo}{{\FF_{q^2}}}
\newcommand{\fqnstar}{{\FF^*_{q^n}}}
\newcommand{\ffive}{{\FF_5}}
\newcommand{\fqnmi}{{\FF_{q^{nm_i}}}}
\newcommand{\fqs}{{\FF_{q^s}}}
\newcommand{\ftwo}{{\FF_2}}
\newcommand{\ftwon}{{\FF_{2^n}}}
\newcommand{\ztwo}{{\ZZ_2}}
\newcommand{\qtwo}{{\QQ_2}}
\newcommand{\fthreen}{{\FF_{3^n}}}
\newcommand{\fthree}{{\FF_3}}
\newcommand{\fthreem}{{\FF_{3^m}}}
\newcommand{\fthreeseven}{{\FF_{3^7}}}
\newcommand{\fthreer}{{\FF_{3^{2^r}}}}

\newcommand{\rmnum}[1]{\romannumeral #1}

\newcommand{\Disc}{\operatorname{Disc}}
\newcommand{\Res}{\operatorname{Res}}
\newcommand{\wt}{\mbox{wt}}
\newtheorem{conj}[theorem]{Conjecture}

\newcommand{\hlinepad}{\raisebox{0pt}[2.5ex]{\strut}}
\newcommand{\Hline}{\hline\hlinepad}
\newcommand{\HHline}{\hline\hline\hlinepad}
\newtheorem{Rem}[theorem]{Remark}

\newcommand{\ie}{i.e., }
\newcommand{\eg}{e.g., }
\newcommand{\ea}{{\em et al. }}


\title{On isogeny classes of Edwards curves over finite fields}

\author{\begin{tabular}[t]{c@{\extracolsep{8em}}c} 
Omran Ahmadi  &  Robert Granger \\
\small{\url{omran.ahmadi@ucd.ie}} & \small{\url{rgranger@computing.dcu.ie}}\\
\end{tabular} 
\\  
\\
\begin{tabular}{c}
        Claude Shannon Institute \\
        University College Dublin\\
        Dublin 4 \\
        Ireland\\
 \end{tabular}
}

\maketitle

\begin{abstract}
We count the number of isogeny classes of Edwards curves over finite fields, answering a question
recently posed by Rezaeian and Shparlinski.
We also show that each isogeny class contains a {\em complete} Edwards curve,
and that an Edwards curve is isogenous to an {\em original}
Edwards curve over $\F_q$ if and only if its group order is divisible by $8$ if $q \equiv -1 \pmod{4}$, 
and $16$ if $q \equiv 1 \pmod{4}$. 
Furthermore, we give formulae for the proportion of $d \in \F_q \setminus \{0,1\}$ for 
which the Edwards curve $E_d$ is complete or original, relative to the total number of $d$
in each isogeny class.
\end{abstract}


\section{Introduction}

In 2007 Edwards proposed a new normal form for elliptic curves over a field $k$ of characteristic $\ne 2$~\cite{edwards}, namely:
\begin{equation}\label{original}
E_{a}(k) : x^2 + y^2 = a^2(1 + x^2 y^2),
\end{equation}
for $a^5 \ne a$. Bernstein and Lange generalised Edwards' form to incorporate curves of the form
$$E(k):x^2 + y^2 = a^2(1+dx^2y^2),$$
which is elliptic if $ad(1-da^4) \ne 0$~\cite{bl}. All curves in the
Bernstein-Lange form are isomorphic to curves of the following form, referred to as
Edwards curves:
\begin{equation}\label{edwards}
E_d(k) : x^2 + y^2 = 1 + d x^2y^2. 
\end{equation}
Edwards curves over finite fields are of great interest in cryptography since the addition and
doubling formulae are: {\em unified}, which protects against some side-channel
attacks~\cite[Chapters 4 and 5]{ecc2}; {\em complete} when $d$ is a non-square,
which means the addition formulae work for all input points; and are the most
efficient in the literature.
Bernstein \ea have also considered {\em twisted} Edwards curves~\cite{twisted}:
\begin{equation}\label{twisted}
E_{a,d}(k):ax^2 + y^2 = 1 + dx^2y^2,
\end{equation}
which includes more curves over finite fields than does Edwards curves.

Rezaeian and Shparlinski have computed the exact number of distinct
curves of the form~(\ref{original}) and~(\ref{edwards}) over a finite
field $\F_q$ of characteristic $> 2$, up to isomorphism over the algebraic
closure of $\F_q$~\cite{shparlinski}. However they state that counting the number of
distinct isogeny classes over $\F_q$ for these curves is a very natural and challenging question.

In this paper we answer this question fully for fields of characteristic $>2$.
Our starting point is interesting in that it was 
serendipitous, beginning with an incidental empirical observation. 
When searching for suitable parameters for elliptic curve cryptography, for curves of the form~(\ref{edwards}),
we observed that over a finite field $\F_p$ with $p \equiv 1 \pmod{4}$,
it (empirically) holds that 
\[
\#E_d(\F_p) = \#E_{1-d}(\F_p),
\]
and hence by Tate's theorem~\cite{tate}, $E_d$ and $E_{1-d}$ should be isogenous over $\F_p$. 

In the course of proving the above observation using character sum identities, we discovered that the Edwards curve $E_d$  
is isogenous to the Legendre curve:
\begin{equation}\label{legendre}
L_d(\fq): y^2 = x(x-1)(x-d).
\end{equation}
With explicit computation one sees that this isogeny has degree two,
and so $E_d$ inherits a set of $4$-isogenies from the well-known set of isomorphisms of $L_d$,
each as the composition of the $2$-isogeny to $L_d$, an isomorphism of $L_{d}$
to $L_{d'}$, and the dual of the $2$-isogeny from $E_{d'}$ to $L_{d'}$.
In particular $E_d/\F_p$ is $4$-isogenous to $E_{1-d}/\F_p$ for $p \equiv 1 \pmod{4}$.
More generally, for $E_d$ over any finite field $\F_q$ one obtains $4$-isogenies
to $E_{1-d}$, $E_{1/d}$, $E_{1-1/d}$, $E_{1/(1-d)}$ and $E_{d/(d-1)}$, 
being defined over $\F_q$ or $\F_{q^2}$ depending on the quadratic character
of $-1,d$ and $1-d$ in $\F_q$.

We later learned that the above $2$-isogeny is merely a special case of Theorem 5.1 of~\cite{twisted}, which
states that any elliptic curve with three $\F_q$-rational $2$-torsion
points is $2$-isogenous to a twisted Edwards curve of the form~(\ref{twisted}). However
the explicit connection with the Legendre curve and the consequent ramifications contained herein
has --- to the best of our knowledge --- not been made before. 

Using the explicit connection with Legendre curves, counting the number of isogeny classes of Edwards curves
is straightforward; we use a recent result due to Katz~\cite{Katz-ratio}, who studied the isogeny 
classes of Legendre curves. In doing so, we also count the number of supersingular parameters $d$ for 
Edwards curves. We then prove the existence of complete Edwards curves in every isogeny class, providing
formulae for the proportion of $d \in \F_q \setminus \{0,1\}$ for 
which $L_d$ --- and hence $E_d$ --- is complete, relative to the total number of $d$ in each
isogeny class. This total be computed via a Deuring-style class number formula derived by
Katz~\cite{Katz-ratio}, and hence for a given trace one can compute the number of
complete Edwards curve parameters $d$. 

We also address the distribution of original Edwards curves~(\ref{original})
amongst the isogeny classes of Edwards curves. For $q \equiv -1 \pmod{4}$ this
follows from our result on complete Edwards curves, but for $q \equiv 1 \pmod{4}$ we express the proportion of
such curves in a given isogeny class using a set of remarkable ratio results due to
Katz~\cite{Katz-ratio}. Whilst we believe our results may be proven succinctly
using a variation of Katz's approach, our arguments for the proportion of complete and original
Edwards curves rely only on explicit bijections between sets of curves of
different parameter types, and are thus entirely elementary.

\vspace{3mm}
\noindent {\em Notation:} For two elliptic curves over a field $k$, we write $E \sim E'$ when $E$ is isogenous to $E'$ over $k$, and 
$E \cong E'$ when $E$ is isomorphic to $E'$ over the algebraic closure of $k$. Throughout the paper, $\F_p$
refers to a finite field of prime cardinality $p$ and $\F_q$ to an extension
field of cardinality $q=p^m$, where $m \geq 1$. Also,
if the field of definition of a curve or map is not specified, it is assumed
to be a field of characteristic $\ne 2$.


\section{A point counting proof of $E_d(\fq) \sim L_d(\fq)$}\label{pointcount}

It is well known that the elliptic integral
$$
\int\frac{p(x)}{\sqrt{q(x)}}dx,
$$ 
where $p(x)\in\R(x)$ is a rational function and $q(x)\in\R[x]$ is a quartic
polynomial, can be reduced to 
$$
\int\frac{p_1(x)}{\sqrt{q_1(x)}}dx
$$ 
for a rational function $p_1(x)\in\R(x)$ and a cubic polynomial $q_1(x)\in\R[x]$ provided that 
one knows one of the roots of $q(x)$~\cite[Chapter 8]{integral}. 

The finite field analogue of this fact is the following result of
Williams~\cite{KSW}.

\begin{lemma}\label{q-c-s}~\cite{KSW} Let $q$ be an odd prime power and let $\F_q$ denote
the finite field with $q$ elements. Suppose that $F(x)$ is a complex valued function
from $\F_q$ to $\C$ and also let $\chi_2(\cdot)$ denote the quadratic character of $\F_q$.
Also let $Z$ denote the zero set of $a_2x^2+b_2x+c_2$. Then
\begin{eqnarray}
\sum_{x\in \F_q \setminus Z} F\left(\frac{a_1x^2+b_1x+c_1}{a_2x^2+b_2x+c_2}\right)&=& 
 \sum_{x\in \F_q}
\chi_2(Dx^2+\Delta x+d)F(x)\\\nonumber&&+\sum_{x\in \F_q} F(x)-\left\{
    \begin{array}{clcl}
    F\left(\frac{a_1}{a_2}\right), & \;\;\;\mbox{if $a_2\neq 0$}, \\\\
    0,&\;\;\;\mbox{otherwise},\\
    \end{array}\right.
\end{eqnarray}
where $a_1,b_1,c_1,a_2,b_2,c_2\in\F_q$,
\begin{equation}
D=b_2^2-4a_2c_2,\; \Delta=4a_1c_2-2b_1b_2+4a_2c_1,\; d=b_1^2-4a_1c_1,
\end{equation}
and
$$\Delta^2-4dD\neq 0.$$ 
\end{lemma}

In the following we use the lemma above to show that $E_d(\fq)$ is isogenous to
$L_d(\fq)$. First notice that the given singular model for Edwards
curves~(\ref{edwards}) has two points at infinity which are singular
and no affine singular points, and resolving the singularities results in four points which
are defined over $\fq$ if and only if $d$ is a quadratic residue in
$\fq$~\cite{bl}. Thus the non-singular model of $E_d(\fq)$ has $2+2\chi_2(d)$
points more than the singular model of $E_d(\fq)$, and 
hence if we rewrite the curve equation of $E_d$ as
\begin{equation}\label{frac}
E_d(\fq): y^2=\frac{x^2-1}{dx^2-1},
\end{equation}
then 
\begin{eqnarray}\label{Edwards-order}
\#E_d(\fq)&=&2+2\chi_2(d)+\sum_{x\in\F_q,\; x\ne \pm d^{1/2}}
(1+\chi_2\left(\frac{x^2-1}{dx^2-1}\right)) \nonumber\\
          &=&2+2\chi_2(d)+q-(1+\chi_2(d))+
\sum_{x\in\F_q,\; x\ne \pm d^{1/2}}\chi_2\left(\frac{x^2-1}{dx^2-1}\right)\nonumber\\
&=&q+1+\chi_2(d)+
\sum_{x\in\F_q,\; x\ne \pm d^{1/2}}\chi_2\left(\frac{x^2-1}{dx^2-1}\right).
\end{eqnarray}

Now on the one hand by applying Lemma~\ref{q-c-s} with $F(x)=\chi_2(x)$, we get
\begin{eqnarray}\label{Ed-Leg}
\sum_{x\in\F_q,\; x\ne \pm d^{1/2}}\chi_2\left(\frac{x^2-1}{dx^2-1}\right)&=& 
\sum_{x\in \F_q}
\chi_2(4dx^2-(4+4d)x+4)\chi_2(x) \nonumber \\
\nonumber &+& \sum_{x\in \F_q} \chi_2(x)-\chi_2(d)\\
&=&\sum_{x\in \F_q} \chi_2((x-1)(dx-1))\chi_2(x)-\chi_2(d)\nonumber\\
&=&\sum_{x\in \F_q} \chi_2(x(x-1)(x-d))-\chi_2(d),
\end{eqnarray}
and on the other hand we have 
\begin{equation}\label{Legendre-order}
\#L_d(\fq)=q+1+\sum_{x\in\F_q}\chi_2(x(x-1)(x-d)),
\end{equation}
where $-\sum_{x\in\F_q}\chi_2(x(1-x)(x-d))$ is the trace of the Frobenius endomorphism.
Thus comparing~(\ref{Edwards-order}),~(\ref{Ed-Leg}),~(\ref{Legendre-order}) we have:
\begin{theorem}\label{mainisog}
The Edwards curve $E_d(\fq)$ and Legendre curve $L_d(\fq)$ are isogenous. 
\end{theorem}

Lemma~\ref{q-c-s} can be viewed as a means of establishing isogeny relations between curves 
defined by relations such as 
\begin{equation}\label{quadratio}
y^2 = \frac{a_1x^2+b_1x+c_1}{a_2x^2+b_2x+c_2},
\end{equation}
and curves defined by $y^2=x(Dx^2+\Delta x+d)$. In \S\ref{quad} we show how to derive
an addition law for curves of the form~(\ref{quadratio}) and prove results similar to those presented in the
intervening sections.


\section{$4$-isogenies of $E_d$}

In this section we detail how to compute explicit $4$-isogenies for $E_d$,
starting with the $2$-isogeny from $E_d$ to $L_d$ and its dual. We then detail 
the well-known isomorphisms of $L_d$ and compose these maps to form the desired $4$-isogenies.

\subsection{Explicit $2$-isogeny $\psi_d:E_d \rightarrow L_d$}

We now derive a $2$-isogeny from $E_d$ to $L_d$, as presented in
the following result.

\begin{theorem}\label{2isogeny}
Let $(x,y) \in E_d$. Then $\psi_d:E_d \rightarrow L_d$ 
\[
(x,y) \mapsto \left( \frac{1}{x^2}, \frac{y(d-1)}{x(1-y^2)}\right).
\]
is a 2-isogeny.
The dual of $\psi_d$ is $\widehat{\psi}_d:L_d \rightarrow E_d:$
\[
(x,y) \mapsto \left( \frac{2{y}}{d-{x}^2}, \frac{{y}^2 - {x}^2(1-d)}{{y}^2 + {x}^2(1-d)}\right).
\]
\end{theorem}

Note that $\psi_d$ is defined on all points of $E_d$ except the kernel elements
$(0,\pm1)$, which map to $\mathcal{O} \in L_d$. 

\begin{proof}
One has the following birational transformation $\tau$ 
\[
\tau(x,y) = \left( (1-d) \frac{1+y}{1-y}, (1-d) \frac{2(1+y)}{x(1-y)}\right), 
\]
from $E_d$ to the Weierstrass curve 
\[
W_d : y^2 = x^3 + 2(1+d)x^2  + (1-d)^2 x,
\]
with inverse
\[
\tau^{-1}(x,y) = \left( \frac{2x}{y}, \frac{x - (1-d)}{x + (1-d)}\right).
\]
While $\tau$ is not defined for the points $(0,\pm 1) \in E_d$, one obtains an
everywhere-defined isomorphism between the respective desingularized
projective models by sending $(0,1)$ to $\mathcal{O} \in W_d$ and $(0,-1)$ to $(0,0)$. Similarly,
$\tau^{-1}$ is not defined at points $(x,y) \in W_d$ satisying $y(x+1-d) = 0$, but if $d$ is a square the points other 
than $(0,0)$ map to points of order $2$ and $4$ at infinity on the
desingularisation of $E_d$ (see the discussion on exceptional points after
Theorem 3.2 of~\cite{twisted}).
The $2$-isogeny used in the proof of Theorem 5.1 of~\cite{twisted} now maps
$W_d$ directly to $L_d$ via 
\[
\phi_d(x,y) = \left( \frac{y^2}{4x^2}, \frac{y((1-d)^2 - x^2)}{8x^2}\right),
\]
with dual 
\[
\widehat{\phi}_d(x,y) = \left( \frac{y^2}{x^2}, \frac{y(d - x^2)}{x^2}\right).
\]
One can verify that the compositions $\phi_d \circ \tau$ and $\widehat{\tau} \circ \widehat{\phi}_d$ give the
stated $\psi_d$ and $\widehat{\psi}_d$ respectively.
\end{proof}

\subsection{Isomorphisms of $L_d$}

The set of isomorphisms of $L_d$ are induced by the two involutions
$\sigma_1(d) = 1-d$ and $\sigma_{2}(d)=1/d$, which induce the following maps from
 $L_d$ to $L_{1-d}$ and $L_{1/d}$ respectively:
\begin{eqnarray}
\label{s1} &\sigma_1:& L_d \longrightarrow L_{1-d}: (x,y) \mapsto (1-x,\sqrt{-1}y),\\
\label{s2} &\sigma_2:& L_d \longrightarrow L_{1/d}: (x,y) \mapsto (x/d,y/d^{3/2}).
\end{eqnarray}
As transformations acting on a given field, the group generated by $\sigma_1,\sigma_2$ is:
\[
H = \{1,\sigma_1,\sigma_2, \sigma_1\sigma_2,\sigma_2\sigma_1, \sigma_1\sigma_2\sigma_1\},
\]
which is isomorphic to the symmetric group $S_3$. The orbit of $d\neq 0,1$ under the action of
$H$ is 
\begin{equation}\label{isovalues}
\Big\{ d, 1-d, \frac{1}{d}, 1-\frac{1}{d}, \frac{1}{1-d}, \frac{d}{d-1} \Big\}
\end{equation}
which has 6 distinct elements provided that
$d$ is
not a root of $d^2 - d + 1=0$ or $(d+1)(d-2)(2d-1)=0$.
Hence we have isomorphisms between each pair of $L_d$, $L_{1-d}$,
$L_{1/d}$, $L_{1- 1/d}$, $L_{1/(1-d)}$ and $L_{d/(d-1)}$. For
completeness we give here the remaining three isomorphisms from $L_d$ to
$L_{\sigma(d)}$ not listed in~(\ref{s1}),(\ref{s2}):
\begin{eqnarray}
\label{s12} &\sigma_1\sigma_2:& L_d \longrightarrow L_{1-\frac{1}{d}}: (x,y) \mapsto (1-x/d,\sqrt{-1}y/d^{3/2}),\\
\label{s21} &\sigma_2\sigma_1:& L_d \longrightarrow L_{\frac{1}{1-d}}: (x,y) \mapsto \left(\frac{1-x}{1-d},\frac{\sqrt{-1}y}{(1-d)^{3/2}}\right),\\
\label{s121} &\sigma_1\sigma_2\sigma_1:& L_d \longrightarrow L_{\frac{d}{d-1}}: (x,y) \mapsto \left(\frac{x-d}{1-d},-\frac{y}{(1-d)^{3/2}}\right).
\end{eqnarray}

\subsection{$4$-isogenies of $E_d$ to $E_{\sigma(d)}$}

Let $\sigma \in H$. Then $\omega_{\sigma(d)}: E_d \rightarrow E_{\sigma(d)}$ is obtained as the following composition:
\[
\omega_{\sigma(d)} = \widehat{\psi}_{\sigma(d)} \circ \sigma \circ \psi_d.
\]
The $2$-isogeny $\widehat{\psi}_{\sigma(d)}$ can be obtained by taking $\widehat{\psi}_d$ and substituting $\sigma(d)$ for $d$.
We do not write down all possible $4$-isogenies but note that whether each is
defined over $\F_q$ or $\F_{q^2}$ is dependent upon the 
quadratic character of $-1, d$ and $1-d$, as determined by maps~(\ref{s1}--\ref{s121}). 
For example, for $q \equiv 1 \pmod{4}$ one has $\chi_2(-1) = 1$ and so $\sigma_1$ is defined over $\F_q$ and
$E_d \sim E_{1-d}$, which was our original observation. We note that the duals
of each of these isogenies are also easily computed.

\subsection{$4$-isogenies of twisted Edwards curves}

One can also map twisted Edwards curves~(\ref{twisted}) to a Legendre form
curve, as given by the following theorem, the proof of which is the same as
the proof of Theorem~\ref{2isogeny}, one having first applied the isomorphism $E_{a,d} \rightarrow
E_{d/a}: (x,y) \mapsto (\sqrt{a}x,y)$.

\begin{theorem} Let $(x,y) \in E_{a,d}$. Then $\psi_{a,d}:E_{a,d} \rightarrow L_{d/a}:$ 
\[
(x,y) \mapsto \left( \frac{1}{ax^2}, \frac{y(d-a)}{a^{3/2}x(1-y^2)}\right).
\]
The dual of $\psi_{a,d}$ is $\widehat{\psi}_{a,d}:L_{d/a} \rightarrow E_{a,d}:$
\[
(x,y) \mapsto \left( \frac{2\sqrt{a}y}{d-{ax}^2}, \frac{{ay}^2 - {x}^2(a-d)}{{ay}^2 + {x}^2(a-d)}\right).
\]
\end{theorem}

One therefore obtains a set of $4$-isogenies from the isomorphisms of
$L_{d/a}$, exactly as before.


\section{Isomorphisms from $L_d$ to Edwards curves}

In addition to the above $2$-isogeny between $E_d$ and $L_d$, one can also
consider when $L_d$ is birationally equivalent to an Edwards curve, \ie is
isomorphic to an Edwards curve. Such isomorphisms have two immediate
consequences. Firstly, for each such isomorphism one obtains a $2$-isogeny of
$E_d$ to another Edwards curve $E_{d'}$ via the composition of $\psi_d$ and 
the isomorphism, see \S\ref{2isog}. Secondly, one is able to deduce the set of Edwards curves
isomorphic to $E_d$, see \S\ref{Eiso}.

\subsection{Isomorphisms from $L_d$ to $E_{\bar{d}}$}\label{2isog}

Since $L_d:y^2 = x^3 - (1+d)x^2 + dx$, one can transform $L_d$ to the Montgomery
curve 
$$M_{A,B}: By^2=x^3+Ax^2+x$$ 
with $A= - (1+d)/\sqrt{d}$, $B = 1/d\sqrt{d}$ via 
$(x,y) \mapsto (x/\sqrt{d},y)$. Using Theorem 3.2 of~\cite{twisted} one then
obtains 
\[
\left( \frac{x}{y\sqrt{d}}, \frac{x-\sqrt{d}}{x+\sqrt{d}} \right) \in
E_{-d(1-\sqrt{d})^2,-d(1+\sqrt{d})^2}, 
\]
which is isomorphic to $E_{\bar{d}}$ with $\bar{d} = \left( \frac{1 + \sqrt{d}}{1 - \sqrt{d}}\right)^2$
with
\[
\rho_d : L_d \rightarrow E_{\bar{d}} : (x,y) \mapsto \left( \sqrt{-1}(1-\sqrt{d})\frac{x}{y}, \frac{x-\sqrt{d}}{x+\sqrt{d}} \right).
\]
Taking the negative root of $d$ in the above transformations gives a second isomorphism, which together we
write as
\[
\rho_{d,\pm} : L_d \rightarrow E_{\bar{d}^{\pm 1}} : (x,y) \mapsto \left(
\sqrt{-1}(1 \mp \sqrt{d})\frac{x}{y}, \frac{x \mp \sqrt{d}}{x \pm \sqrt{d}} \right).
\]
We also have
\[
\widehat{\rho}_{d,\pm}: E_{\bar{d}^{\pm 1}} \rightarrow L_d : (x,y) \mapsto
\left( \pm \sqrt{d}\frac{1+y}{1-y}, 
\pm \sqrt{-1}\sqrt{d}(1 \mp \sqrt{d})\frac{1+y}{x(1-y)} \right).
\]
Clearly these isomorphisms are only defined over the ground field 
if both $-1$ and $d$ are quadratic residues. 

Observe that the value $\bar{d}$ is invariant under the substitution $d
\leftarrow 1/d$, hence the $L_d$-isomorphic curve $L_{1/d}$ maps to
$E_{\bar{d}}$ also, but with the $\pm$ isogenies defined instead by 
\[
\rho_{1/d,\pm} : L_{1/d} \rightarrow E_{\bar{d}^{\pm 1}} : (x,y) \mapsto
\left( \sqrt{-1}(1 \mp \sqrt{1/d})\frac{x}{y}, \frac{x \mp \sqrt{1/d}}{x \pm \sqrt{1/d}} \right),
\]
with inverse $\widehat{\rho}_{1/d,\pm}: E_{\bar{d}^{\pm 1}} \rightarrow L_{1/d}:$
\[
(x,y) \mapsto \left( \pm \sqrt{1/d}\frac{1+y}{1-y}, 
\pm \sqrt{-1}\sqrt{1/d}(1 \mp \sqrt{1/d})\frac{1+y}{x(1-y)} \right).
\] 

Similarly, one can first map $L_d$ to $L_{\sigma(d)}$ for any $\sigma \in H$, and then apply $\rho_{d,\pm}$ but with the substitution 
$d \leftarrow \sigma(d)$ to give $\theta_{\sigma(d),\pm}: L_{d} \rightarrow
L_{\sigma(d)} \rightarrow E_{\overline{\sigma(d)}^{\pm 1}}$.
We thus have twelve isomorphisms $\theta_{\sigma(d),\pm}$ from $L_d$ to the six curves
$E_{\bar{d}_{i}^{\pm 1}}$ for $i \in \{1,2,3\}$, with:
\[
\bar{d}_{1}^{\pm 1} = \left( \frac{1 \pm \sqrt{d}}{1 \mp \sqrt{d}}\right)^2,
\bar{d}_{2}^{\pm 1} = \left( \frac{1 \pm \sqrt{1-d}}{1 \mp \sqrt{1-d}}\right)^2 \ \text{and} \ 
\bar{d}_{3}^{\pm 1} = \left( \frac{1 \pm \sqrt{\frac{d}{d-1}}}{1 \mp \sqrt{\frac{d}{d-1}}}\right)^2. 
\]
As noted above the twelve isomorphisms have only the six image 
curves $E_{\bar{d}_{1}^{\pm 1}}, E_{\bar{d}_{2}^{\pm 1}}$ and
$E_{\bar{d}_{3}^{\pm 1}}$, since $d$ and $1/d$ map to $\bar{d}_1$, $1-d$ and
$1/(1-d)$ map to $\bar{d}_2$, and $d/(d-1)$ and $1-1/d$ map to $\bar{d}_3$.
These curves are therefore isomorphic and each has $j$-invariant
\[
\frac{2^8(d^2 - d + 1)^3}{(d(d-1))^2},
\]
which is the Legendre curve $j$-invariant $j_L(d)$.

Taking the composition of $\psi_d$ and an isomorphism from each of the six
pairs of isomorphisms above --- one
from each pair that have the same image --- one obtains 
$2$-isogenies of $E_d$ to $E_{\bar{d}_{1}^{\pm 1}}, E_{\bar{d}_{2}^{\pm 1}}$ and
$E_{\bar{d}_{3}^{\pm 1}}$, again defined over $\F_{q}$ or $\F_{q^2}$ depending
on the quadratic charcter of $-1,d$ and $1-d$,
which we summarise in Theorem~\ref{2isogenies}. 
We note that Moody and Shumow have independently given equivalent isogenies~\cite{moody}, having 
obtained them using a different approach.

\begin{theorem}\label{2isogenies}
There exist $2$-isogenies of $E_d$ to $E_{\bar{d}_{1}^{\pm 1}}, E_{\bar{d}_{2}^{\pm 1}}$ and
$E_{\bar{d}_{3}^{\pm 1}}$, given by the following maps, respectively:
\begin{itemize}
\item[(a)] $\epsilon_{\bar{d}_{1},\pm} : E_d \rightarrow
  E_{\bar{d}_{1}^{\pm 1}}: (x,y) \mapsto \left( \frac{\sqrt{-1}(1 \mp \sqrt{d})}{d-1}
  \frac{1-y^2}{xy}, \frac{1 \mp \sqrt{d}x^2}{1 \pm \sqrt{d}x^2} \right)$,
\item[(b)] $\epsilon_{\bar{d}_{2},\pm} : E_d \rightarrow
  E_{\bar{d}_{2}^{\pm 1}}: (x,y) \mapsto \left( (1 \mp \sqrt{1-d})xy,
  \frac{1-(1 \mp \sqrt{1-d})x^2}{1-(1 \pm \sqrt{1-d})x^2} \right)$,
\item[(c)] $\epsilon_{\bar{d}_{3},\pm} : E_d \rightarrow
  E_{\bar{d}_{3}^{\pm 1}}: (x,y) \mapsto \left( \frac{\sqrt{d-1} \mp
    \sqrt{d}}{1-d}\frac{x}{y}, \frac{1-\big(d \pm (1-d)\sqrt{\frac{d}{d-1}}\big)x^2}{1-\big(d \mp (1-d)\sqrt{\frac{d}{d-1}}\big)x^2}\right)$.
\end{itemize} 
\end{theorem}

Theorem~\ref{2isogenies} allows one to write down the set of $4$-isogenies between $E_d$ and any
$E_{\sigma(d)}$ via isogenies and isomorphisms of Edwards curves only: first
map $E_d \rightarrow E_{\bar{d}_{i}^{\pm 1}}$; second apply an
isomorphism to the relevant $E_{\bar{d}_{j}^{\pm 1}}$; and third use a dual isogeny
to map to $E_{\sigma(d)}$. However, since the Edwards $2$-isogenies implicitly depend
on the $2$-isogeny to $L_d$, the initial derivation given is perhaps the most
natural way to view these $4$-isogenies.

\subsection{Isomorphisms of $E_d$}\label{Eiso} 

It is clear from \S\ref{2isog} that the $E_{\bar{d}_{i}^{\pm 1}}$ curves
inherit isomorphisms from the isomorphisms of $L_d$, whereas $E_d$ inherits
isogenies from the isomorphisms of $L_d$ 
--- in both instances $L_d$ plays a fundamental role. A natural question is whether or not it is possible to exploit the isomorphisms between 
$E_{\bar{d}_{i}^{\pm 1}}$ to give the set of curves isomorphic to $E_d$?
Since the $j$-invariant of $E_d$ is~\cite{twisted}
\[
j_E(d) = \frac{16(d^2 + 14d +1)^3}{d(d-1)^4},
\]
it would not seem obvious how to determine the set of isomorphic curves of
$E_d$ from those of $L_d$. However, one can argue as follows. 
As above let $\delta = \bar{d}_1(d) = \left( \frac{1 + \sqrt{d}}{1 -
  \sqrt{d}}\right)^2$, with $\bar{d}_1$ considered as a function of
$d$. Observe that $d = (\bar{d}_{1}(\delta))^{-1}$, and hence
\[
E_d = E_{\left( \frac{1-\sqrt{\delta}}{1+\sqrt{\delta}} \right)^2}.
\]
Since the curve on the right-hand-side is isomorphic to $E_{\bar{d}_{1}^{\pm 1}(\delta)}$,
$E_{\bar{d}_{2}^{\pm 1}(\delta)}$ and $E_{\bar{d}_{3}^{\pm 1}(\delta)}$, so is
$E_d$. Writing these expressions out in full gives the following theorem.
\begin{theorem}\label{jinv}
Let $E_d$ and $E_{d'}$ be two Edwards curves. Then $E_{d} \cong E_{d'}$ if and
only if 
\[ 
d' \in \Bigg\{d,1/d, \left(\frac{1 \pm d^{1/4}}{1 \mp d^{1/4}}\right)^4,
\left(\frac{1 \pm \sqrt{-1}d^{1/4}}{1 \mp \sqrt{-1}d^{1/4}}\right)^4
\Bigg\}.
\]
\end{theorem}
These six values are naturally implied by Proposition 6.1 of Edwards original exposition~\cite{edwards}.
In particular curve (1) is isomorphic to curve (2) via the map $(x,y) \mapsto (ax,ay)$, with $d = a^4$.
Taking the fourth power of each of the $24$ values given in Edwards' proposition gives the six values listed in Theorem~\ref{jinv}.
It is however interesting that these values can be determined from the isomorphisms of $L_d$ alone.
The above manipulations also show that $E_d \cong L_{\delta}$, via 
\[
(x,y) \mapsto  \left( \frac{\sqrt{d} + 1}{\sqrt{d} - 1}\cdot \frac{1+y}{1-y},
\frac{2\sqrt{-1}(1+\sqrt{d})}{(1-\sqrt{d})^2}\cdot\frac{1+y}{x(1-y)} \right).
\]
Note that the existence of such an isomorphism is implied by the fact that $j_L(\delta) = j_E(d)$.


\section{The number of isogeny classes of Edwards curves over finite fields}

In this section we derive some results about Edwards
curves~(\ref{edwards}), from results known for the Legendre family of elliptic
curves, which is well-studied. Having established the isogeny between $E_d$ and $L_d$ in
Theorem~\ref{2isogeny}, the validity of this approach is immediate.
In particular we determine the number of isogeny classes of Edwards
curve over the finite field $\fq$, and in the course of doing so also detail
the number of supersingular curves $E_d(\fq)$. 

For the Legendre curve $L_d(\fq)$, 
we denote the trace of the Frobenius endomorphism 
\begin{equation}
-\sum_{x\in\F_q}\eta(x(x-1)(x-d))
\end{equation}
by $A(d,\fq)$. Then Equation~(\ref{Legendre-order}) implies
\begin{equation}\label{Leg-order-2}
\#L_d(\fq)=q+1-A(d,\fq),
\end{equation}
and by the Hasse-Weil bound we have
$$
\left|A(d,\fq)\right|\le 2\sqrt{q}.
$$
Thus the number of isogeny classes of the Legendre family of elliptic curves 
is the same as the number of integer values of $A$ with 
$\left|A\right|\le 2\sqrt{q}$
for which there is a $d$ such that $A(d,\fq)=A$. The following two lemmata
give a satisfactory answer to this question. The first addresses the number 
of ordinary isogeny classes and the second addresses the supersingular isogeny classes.
\begin{lemma}\cite{Katz-ratio}\label{ordinary}
Let $\fq$ be a finite field of odd characteristic, and let $A\in\Z$ 
be an integer prime to $p$
(the characteristic of $\fq$) with $\left|A\right|\le 2\sqrt{q}$.
If $A\equiv q+1 \pmod 4$, then there exists $d\in\fq\backslash\{0,1\}$
with $A(d,\fq)=A$. 
\end{lemma}

\begin{lemma}\cite{Katz-ratio}\label{supersingular}
Let $p$ be an odd prime. Then we have the following assertions. 
\begin{itemize}
\item[(i)] If $q=p^{2k+1}$, and $L_d(\fq)$ is 
supersingular, then $A(d,\fq)=0$.
\item[(ii)] If $q=p^{2k}$, and $L_d(\fq)$ is 
supersingular, then $A(d,\fq)=\epsilon 2p^k$, where $\epsilon=\pm 1$ is 
the choice of sign for which $\epsilon p^k\equiv 1 \pmod 4$.
\end{itemize}
\end{lemma}

Following Katz, we say that each $A$ satisfying the conditions of
Lemma~\ref{ordinary} is {\em unobstructed}, for $q$. From the two lemmata above, the following is immediate.
\begin{cor}
If $q=p^{2k+1}$ and $p\equiv 1 \pmod 4$, then the number of isogeny classes
of Edwards curves over $\fq$ is 
$$
2\left\lfloor\frac{\lfloor 2\sqrt{q}\rfloor+2}{4}\right\rfloor-
2\left\lfloor\frac{\left\lfloor\frac{\lfloor 2\sqrt{q}\rfloor}{p}\right\rfloor+2}{4}\right\rfloor.
$$

\end{cor} 
\begin{proof}
The claim will follow if we prove that there is no supersingular Legendre curve
in this case. Observe that $\#L_d(\fq)$ is always divisible by 4, and if $q=p^{2k+1}$, $p\equiv 1 \pmod 4$
and $L_d(\fq)$ is supersingular, then from  
Lemma~\ref{supersingular}(i) and~(\ref{Leg-order-2}) it follows that  $\#L_d\equiv 2\pmod 4$, which is impossible. 
\end{proof}

In order to obtain the number of isogeny classes of Edwards curves in the remaining cases
we need to know how the supersingular Legendre curve parameters are
distributed amongst extensions of the prime subfield $\F_p$ of $\F_q$; again, 
there is already a complete answer to this question in the literature. On the one hand, it is 
well known that $L_d(\fq)$ is a supersinular curve if and only if $d$ is a root of 
the Hasse-Deuring polynomial
\[
H_p(x)=(-1)^{\frac{p-1}{2}}\sum_{i=0}^{\frac{p-1}{2}}{{(p-1)/2}\choose{i}}^2x^i,
\]   
and on the other hand it is well known that all the roots of Deuring polynomial are in $\F_{p^2}$
(see for example \cite[Proposition~2.2]{Auer-Top}).
Using Theorem~\ref{2isogeny} and \cite[Proposition~3.2]{Auer-Top} the following is immediate.

\begin{theorem}\label{supersingular-p}
The number $S_p$ of $\F_p$-rational roots of the Deuring polynomial, or equivalently the number
of supersingular Edwards curves over $\F_p$, satisfies
\begin{itemize}
\item[(i)] $S_p=0$ if and only if $p\equiv 1 \pmod 4$.

\item[(ii)] $S_3=1$.

\item[(iii)] If $p\equiv 3\pmod 4$ and $p> 3$, then $S_p=3h(-p)$, where $h(-p)$ is the 
class number of $\QQ(\sqrt{-p})$.

\end{itemize} 
\end{theorem}

\begin{cor}
If $p\equiv 3 \pmod 4$ and $q=p^{2k+1}$, then the number of isogeny classes
of Edwards curves over $\fq$ is 
$$
2\left\lfloor\frac{\lfloor 2\sqrt{q}\rfloor}{4}\right\rfloor-
2\left\lfloor\frac{\lfloor 2\sqrt{q}\rfloor}{4p}\right\rfloor+1.
$$
\end{cor}
\begin{proof}
From Lemma~\ref{supersingular} and Theorem~\ref{supersingular-p} it follows that
there is a single isogeny class of supersingular Legendre curves in this case.
\end{proof}

Similarly we have:
\begin{cor}
If $q=p^{2k}$ for an odd prime $p$, then the number of isogeny classes
of Edwards curves over $\fq$ is 
$$
2\left\lfloor\frac{\lfloor 2\sqrt{q}\rfloor+2}{4}\right\rfloor-
2\left\lfloor\frac{\left\lfloor\frac{\lfloor 2\sqrt{q}\rfloor}{p}\right\rfloor+2}{4}\right\rfloor+1.
$$
\end{cor}
\begin{proof}
From the fact that all the roots of Hasse-Deuring polynomial are in $\fptwo$
and from Lemma~\ref{supersingular}, it follows that there is
a single isogeny class of supersingular Legendre curves in this case.
\end{proof}

\section{Isogeny classes of {\em complete} Edwards Curves}\label{complete}

Bernstein and Lange proved that the Edwards addition law is complete, \ie is well-defined
on all inputs, if and only if $\chi_2(d) = -1$~\cite{bl}. 
A natural question to consider is whether there exists a complete Edwards curve
in every isogeny class. In this section we answer this question affirmatively,
relating the number of non-square $d \in \F_q \setminus \{0,1\}$ in each
isogeny class to the total number of $d$ in each isogeny class.

\subsection{Katz's ratio results} 

While investigating the Lang-Trotter conjecture~\cite{LT}, Katz discovered some remarkable relationships between 
the number of $d \in \F_q \setminus \{0,1\}$ such that $A(d,\F_q) = q + 1 -
\#L_d = A$ for any unobstructed $A$, and the number of $d \in \F_q \setminus
\{0,1\}$ such that $A(d,\F_q) = -A$~\cite{Katz-ratio}. 

In particular, let $N(A) = \#\{ d \in \F_{q} \setminus \{0,1\} | A(d,\F_q) = A\}$. 
Katz proved that for $q \equiv -1 \pmod{4}$, one has
$N(A) = N(-A)$. For $q \equiv 1 \pmod{4}$, this is no longer the case.
Since $A \equiv 2 \pmod 4$, exactly one of $A,-A$ has $q+1 -A \equiv 0 \pmod 8$ --- call it $A$ --- with $q+1+A \equiv 4 \pmod{8}$.
Then $N(A) > N(-A)$. Furthermore, for $q \equiv 5 \pmod{8}$ the ratio
$r = N(A)/N(-A)$ is always one of the integers $2,3$, or $5$, 
depending only on the power of $2$ dividing $q+1-A$, as given in:

\begin{theorem}\cite[Theorem 2.8]{Katz-ratio}\label{katz2.8}
Suppose $q \equiv 5 \pmod{8}$. Then
\begin{eqnarray}
\nonumber ord_2(q+1-A) = 3 \Longrightarrow r = 2,\\
\nonumber ord_2(q+1-A) = 4 \Longrightarrow r = 3,\\
\nonumber ord_2(q+1-A) \ge 5 \Longrightarrow r = 5.
\end{eqnarray}
\end{theorem}

For $q \equiv 1 \pmod{8}$ the situation is more complicated. If
$ord_2(q+1-A)=3$ then $r = 2$ as before. Let $\Delta = A^2 - 4q$. For the remaining cases we have:

\begin{theorem}\cite[Theorem 2.11]{Katz-ratio}\label{katz2.11}
Suppose $q \equiv 1 \pmod{8}$, and that $ord_2(q+1-A) \ge 4$. Then
$ord_2(\Delta) \ge 6$, and we have the following results.
\begin{itemize}
\item[(1)] Suppose $ord_2(\Delta) = 2k+1, k \ge 3$. Then $r = 5 - 3/2^{k-2}$.
\item[(2)] Suppose $ord_2(\Delta) = 2k, k \ge 3$. Then
\begin{itemize}
\item[(a)] if $\Delta/2^{2k} \equiv 1 \bmod{8}$, then $r = 5$,
\item[(b)] if $\Delta/2^{2k} \equiv 3 \ \text{or} \ 7 \bmod{8}$, then $r = 5 - 3/2^{k-2}$,
\item[(a)] if $\Delta/2^{2k} \equiv 5 \bmod{8}$, then $r = 5 - 1/2^{k-3}$.
\end{itemize}
\end{itemize}
\end{theorem}

To explain these phenomena, Katz uses the fact that $L_d$ is $2$-isogenous to
the elliptic curve $y^2 = (x+t)(x^2 + x + t)$, $t \ne 0,1/4$, having a point $(0,t)$ of order
$4$ and where $t=(1-d)/4$. Over the $t$-line, this family of
curves with its point $(0,t)$ is the universal curve given with a point of order
$4$. Using this property Katz derives a Deuring-style class number formula 
to express the number of $t \in \F_q$ such that
$A(t,\F_q) = A$. Expressing the same for $-A$ and then computing the
ratio $N(A)/N(-A)$ happens to be far simpler than computing the exact numbers themselves,
as it obviates the need to perform any class group order computations. 
However, in the proof no consideration was given (nor was it needed) of the quadratic
character of elements $t$ in a given $N(A)$. Furthermore, since under this
$2$-isogeny we have $t = (1-d)/4$, determining how the corresponding square
and non-square $d$ are distributed between the numerator and denominator of 
$N(A)/N(-A)$ is certainly not immediate.

However, we observed (empirically - and then proved) that the following holds. Let $N_{2}(A)$
and $N_{n2}(A)$ be the partition of $N(A)$ into square and non-square $d$
respectively, and similarly for $-A$. For $q \equiv 1 \pmod{4}$, we have
$N_{n2}(A) = N_{n2}(-A) = N(-A)$, \ie the smallest of the two values
$N(A),N(-A)$. Hence the excess of $N(A)$
over $N(-A)$ consists entirely of square $d$. 
For $q \equiv -1 \pmod{4}$ we have
\[
N_{n2}(A) = \begin{cases} N(A) \hspace{7mm} \text{if} \ q+1-A \equiv 4 \pmod{8} \\ 
N(A)/3 \hspace{3mm} \text{if} \ q+1-A \equiv 0 \pmod{8}. \end{cases}
\]
Since $q \equiv -1 \pmod 4$ we have $N_{n2}(A) = N_{n2}(-A)$ in this case
also. Our proof of these facts is elementary. 

\subsection{Proof of claims}

We use the following three lemmata, the 
first of which can be found in~\cite[Theorem 8.14]{washington} (see also~\cite[X, Sect. 1]{silverman}):

\begin{lemma}[2-descent]\label{2descent}
Assume $\text{char}(\F_q) > 2$, and let $E(\F_q)$ be given by $y^2 = (x -
\alpha)(x-\beta)(x-\gamma)$ with $\alpha,\beta,\gamma \in \F_q$, $\alpha \ne
\beta \ne \gamma \ne \alpha$. The map
\[
\phi: E(\F_q) \longrightarrow \F_{q}^{\times}/(\F_{q}^{\times})^2 \times \F_{q}^{\times}/(\F_{q}^{\times})^2 \times \F_{q}^{\times}/(\F_{q}^{\times})^2 
\]
defined by
\begin{eqnarray}
\nonumber (x,y) &\mapsto& (x-\alpha, x-\beta, x- \gamma) \ \text{when} \ y\neq 0\\ 
\nonumber \mathcal{O} &\mapsto& (1,1,1)\\
\nonumber (e_1,0) &\mapsto& ((e_1 - e_2)(e_1 - e_3), e_1 - e_2, e_1 - e_3)\\
\nonumber (e_2,0) &\mapsto& (e_2 - e_1, (e_2 - e_1)(e_2 - e_3), e_2 - e_3)\\
\nonumber (e_3,0) &\mapsto& (e_3 - e_1, e_3 - e_2, (e_3 - e_1)(e_3 - e_2))
\end{eqnarray}
is a homomorphism, with kernel $2E(\F_q)$.
\end{lemma}

Applying Lemma~\ref{2descent} to the $2$-torsion points
$(0,0),(1,0)$ and $(d,0)$ of $L_d(\F_q)$, one can compute the possible
$4$-torsion groups $L_d(\F_q)[4]$, which depend only on $\chi_2(-1),\chi_2(d)$ and
$\chi_2(1-d)$, giving the following result.

\begin{lemma}\label{4torsion}
For $q \equiv \pm 1 \bmod{4}$, the possible $4$-torsion
groups $L_d(\F_q)[4]$, are those detailed in Tables 1 and 2 respectively.
\begin{table}[here]
\caption{$q \equiv 1 \bmod{4}$}
\begin{center}
\begin{tabular}{c|c|c|c}
\hline
$\chi_2(d)$     & $\chi_2(1-d)$ & $(L_d(\F_q)[2] \cap 2L_d(\F_q))\setminus\{\mathcal{O}\}$  & $L_d(\F_q)[4]$\\
\hline
$1$     &  $1$    &  $(0,0),(1,0),(d,0)$  & $\Z/4\Z \times \Z/4\Z$ \\
$-1$    &  $1$    &  $(1,0)$              & $\Z/4\Z \times \Z/2\Z$ \\
$1$     & $-1$    &  $(0,0)$              & $\Z/4\Z \times \Z/2\Z$ \\
$-1$    &  $-1$   &                       & $\Z/2\Z \times \Z/2\Z$ \\
\hline
\end{tabular}
\end{center}
\end{table}

\begin{table}[here]
\caption{$q \equiv -1 \bmod{4}$}
\begin{center}
\begin{tabular}{c|c|c|c}
\hline
$\chi_2(d)$     & $\chi_2(1-d)$ & $(L_d(\F_q)[2] \cap 2L_d(\F_q))\setminus\{\mathcal{O}\}$  & $L_d(\F_q)[4]$\\
\hline
$1$     &  $1$    &  $(1,0)$              & $\Z/4\Z \times \Z/2\Z$ \\
$-1$    &  $1$    &  $(1,0)$              & $\Z/4\Z \times \Z/2\Z$ \\
$1$     & $-1$    &  $(d,0)$              & $\Z/4\Z \times \Z/2\Z$ \\
$-1$    &  $-1$   &                       & $\Z/2\Z \times \Z/2\Z$ \\
\hline
\end{tabular}
\end{center}
\end{table}
\end{lemma}

We also use the following easy result, the first part of which was also used by
Katz~\cite[Lemma 2.3]{Katz-ratio}.

\begin{lemma}\label{easyisos}
For $d \in \F_q \setminus \{0,1\}$ we have:
\begin{itemize}
\item[(i)] $A(d,\F_q) = \chi_2(-1) \cdot A(1-d,\F_q)$,
\item[(ii)] $A(d,\F_q) = \chi_2(d) \cdot A(1/d,\F_q)$.
\end{itemize}
\end{lemma}
\begin{proof}
These are immediate consequences of isomorphisms~(\ref{s1}) and~(\ref{s2}).
\end{proof}

We are now ready to prove our observations.

\begin{theorem}\label{1mod4nonsquare}
For $q \equiv 1 \pmod{4}$, let $A$ be such that $q+1 - A \equiv 0
\pmod{8}$ (and so $q+1 + A \equiv 4 \pmod{8}$). Then $N_{n2}(A) = N_{n2}(-A) =
N(-A)$.
\end{theorem}
\begin{proof}
From Table 1 we see that for any square $d$, $L_d(\F_q)$ contains a subgroup
of order either $8$ or $16$. As $q + 1 + A \equiv 4 \pmod{8}$, by
Lagrange's theorem we must have $N_{2}(-A) =0$ . Hence all $d$ counted
by $N(-A)$ are necessarily non-square, and since by Lemma~\ref{ordinary} every unobstructed $A$
occurs, we have $N_{n2}(-A) = N(-A)$. 
Since $\F_{q} \setminus \{0,1,-1\}$ partitions
into a disjoint union of pairs $\{d,1/d\}$, 
by Lemma~\ref{easyisos}(ii) for non-square $d$ we have 
a bijection between the elements counted by $N_{n2}(-A)$ and those counted by $N_{n2}(A)$, and
hence these numbers are equal.
\end{proof}

\begin{theorem}\label{notsquare2}
For $q \equiv -1 \pmod{4}$, we have
\[
N_{n2}(A) = \begin{cases} N(A) \hspace{7mm} \text{if} \ q+1-A \equiv 4 \pmod{8} \\ 
N(A)/3 \hspace{3mm} \text{if} \ q+1-A \equiv 0 \pmod{8}. \end{cases}
\]
\end{theorem}

\begin{proof}
We show that the result is true in each isomorphism class.
First, assume $j_L(d) \ne 0,1728$, so that each isomorphism class contains the 
six distinct elements in~(\ref{isovalues}). From Table 2 we have that for any 
square $d$, $L_d(\F_q)$ contains a subgroup of order $8$. Hence if $\#L_d(\F_q) = q + 1 - A \equiv 4
\pmod{8}$, by Lagrange's theorem we must have $N_{2}(A) = 0$. Hence all $d$ counted
by $N(A)$ are non-square, and since every unobstructed $A$
occurs, we have $N_{n2}(A) = N(A)$. This proves the first part of the theorem.
For the second part, we shall show that for each $A$ for which $q+1-A
\equiv 0 \pmod{8}$, square $d$ occur twice as frequently as non-square 
$d$ in the counts for both $N(A)$ and $N(-A)$. Abusing notation slightly, when $A(d,\F_q) = A$
we write $d \in N(A)$, and simlarly for $N(-A)$.

Let $\#L_d(\F_q) = q+1-A \equiv 0 \pmod{8}$. Then by Sylow's 1st theorem,
$L_d(\F_q)$ contains a subgroup of order $8$, and hence $L_d(\F_q)[8]$
contains at least $8$ points. By Table 2, we can not have
$\chi_2(d)=\chi_2(1-d)=-1$, since $L_d(\F_q)[4] \cong \Z_2 \times \Z_2 = L_d(\F_q)[2]$ and hence
$|L_d(\F_q)[2^i]| = 4$ for $i \ge 2$. Hence we have three possibilities for
$(\chi_2(d),\chi_2(1-d))$.

Let $\chi_2(d)=1$ with $d \in N(A)$. Then by Lemma~\ref{easyisos}(ii), $1/d \in
N(A)$ also. By Lemma~\ref{easyisos}(i), $1-d, 1-1/d \in N(-A)$. If
$\chi_2(1-d)=-1$ then by Lemma~\ref{easyisos}(ii) we have $1/(1-d) \in N(A)$, and 
$d/(d-1) \in N(-A)$. Hence $\{d,1/d,1/(1-d)\} \in N(A)$ and $\{1-d,1-1/d,d/(d-1)\} \in N(-A)$,
and there are two squares and a non-square in each set, as asserted.
If $\chi_2(1-d)=1$ then by Lemma~\ref{easyisos}(ii) we have instead $1/(1-d) \in
N(-A)$, and $d/(d-1) \in N(A)$. Hence $\{d,1/d,d/(d-1)\} \in N(A)$ and 
$\{1-d,1-1/d,1/(1-d)\} \in N(-A)$, and again there are two squares and a non-square in each set.
Finally, if $\chi_2(d)=-1$ and $\chi_2(1-d)=1$, by Lemma~\ref{easyisos} again
we see that if $d \in N(A)$ then $\{d,1-1/d,d/(d-1)\} \in N(A)$ and $\{1/d,1-d,1/(1-d)\} \in N(-A)$. 
In all cases $N_{2}(A) = 2N_{n2}(A)$ and $N_{2}(-A) = 2N_{n2}(-A)$, and the second part of the
result follows for these isomorphism classes.

If $j_L(d) = 1728$, \ie if $d = 2,1/2,-1$, it is easy to see that Lemma~\ref{easyisos}
implies that the trace of Frobenius is zero in all cases. Now $\chi_2(2) = -1$ if $q \equiv
3 \pmod 8$ and is $1$ if $q \equiv 7 \pmod{8}$. In the first case, $q+1 - 0
\equiv 4 \pmod{8}$ and this isomorphism class contributes three elements to
$N_{n2}(0)$ and hence $N(0)$.
In the second case $q+1 - 0 \equiv 0 \pmod{8}$ and this class contributes two
squares and one non-square to $N(0)$.

If $j_L(d) =0$ then $d^2-d+1 =0$, \ie $d$ and $1/d$ are primitive $6$-th roots
of unity over $\F_q$, which are in $\F_q$ iff $q \equiv 1 \pmod{6}$. Since $q \equiv -1 \pmod{4}$ we
must have $q \equiv 7 \pmod{12}$. In particular, $\F_q$ does not contain any
$12$-th roots of unity and hence $\chi_2(d) = -1$. 
Since $1-d = 1/d$, we have $\chi_2(1-d) = \chi_2(1/d) = -1$, and so by Table 2, $L_d(\F_q)[4] \cong \Z_2
\times \Z_2$ and hence $\#L_d(\F_q) = q+1-A \equiv 4 \pmod{8}$ by the above argument. 
By Lemma~\ref{easyisos}(ii), $A(d, \F_q) = - A(1/d,\F_q)$ and this isomorphism
class contributes one element to $N_{n2}(A)$ and hence $N(A)$, and
one element to $N_{n2}(-A)$ and hence $N(-A)$, whenever this isomorphism class is
defined over $\F_q$.
\end{proof}

Since by Lemma~\ref{ordinary} we have $N(A) > 0$ for every unobstructed
integer $A$ for a given $q$, we thus have the following.

\begin{cor}\label{cor}
Let $A$ be an unobstructed integer for $q$. Then there exists at least one quadratic non-residue $d \in \F_{q}
\setminus \{0,1\}$ such that $\#E_d(\fq) = q + 1 - A$, and hence there is a
complete Edwards curve in every isogeny class.
\end{cor}

Theorems~\ref{1mod4nonsquare} and~\ref{notsquare2} allow one can compute $N_{n2}(A)$ given $N(A)$, which
 can be computed using Katz's Deuring-style class number formula~\cite{Katz-ratio}. In fact for 
$q \equiv 1 \pmod{4}$, the formula for $N(-A)$ is far simpler than that for $N(A)$, while for $q \equiv -1 \pmod{4}$,
$N(A)$ and $N_{n2}(A)$ are either equal or differ by a factor of $3$.

To conclude this section, we note that Morain has independently proven the following~\cite[Theorem 17]{morain}.

\begin{theorem}\label{morain}
Let $E(\F_p): y^2 = x^3 + a_2x^2 + a_4x + a_6$ have three $\F_p$-rational
$2$-torsion points. Then there exists a curve $E'(\F_p)$ isogenous to $E(\F_p)$ that
is birationally equivalent to a complete Edwards curve.
\end{theorem}

Therefore, if such a curve $E(\F_q)$ exists in every isogeny class
whose group order is necessarily divisible by $4 = |E(\F_q)[2]|$, then
Theorem~\ref{morain} implies Corollary~\ref{cor}; Theorem~\ref{mainisog}
provides the missing condition. Furthermore, Morain's proof is constructive, in that from such a curve $E$ one can 
explicitly compute a set of isomorphism classes of complete Edwards curves, 
based on the structure of the volcano of $2$-isogenies of $E$.


\section{Isogeny classes of {\em original} Edwards curves}

As stated in \S\ref{Eiso}, curves in Edwards' original normal form~(\ref{original}) are
isomorphic to the Bernstein-Lange form~(\ref{edwards}) via $(x,y) \mapsto
(ax,ay)$, with $d = a^4$. Two natural questions to consider are whether or not there exists an original
Edwards curve in every isogeny class, and more specifically how are the original Edwards curves
distributed amongst the isogeny classes? In this section we present answers to
both these questions.

We begin with some definitions. For any unobstructed $A$ for $q$, let $N_{4}(A)$
and $N_{2n4}(A)$ be the number of $d \in N(A)$ that are fourth powers, and
squares but not fourth powers, respectively. For any such $A$ we thus have
\begin{equation}\label{partition}
N(A) = N_{n2}(A) + N_{2n4}(A) + N_4(A).
\end{equation}
Furthermore let $\chi_4(\cdot)$ denote a primitive biquadratic character of $\F_q$,
so that $\chi_4(d) = 1$ if and only if there exists an $a \in \F_q$ such that
$d = a^4$.

\subsection{Determining $L_d(\F_q)[8]$}

In the ensuing treatment, we will need to know the possible $8$-torsion subgroups of $L_d(\F_q)$.
The structure of the $4$-torsion was determined by analysing the halvability
of the $2$-torsion points, using Lemma~\ref{2descent}. Similarly, one can
apply Lemma~\ref{2descent} to the elements of $L_d(\F_q)[4] \setminus
L_d(\F_q)[2]$ to determine the structure of the $8$-torsion.

Over the algebraic closure of $\F_q$ there are twelve points of order
four; two for each of the three $2$-torsion points $(0,0),(1,0)$ and $(d,0)$:
\begin{eqnarray*}
P_{(0,0),\pm} &=& (\pm \sqrt{d},\sqrt{-1} \sqrt{d}(1 \mp \sqrt{d})),\\ 
P_{(1,0),\pm} &=& (1 \pm \sqrt{1-d},\sqrt{1-d}(1 \pm \sqrt{1-d})),\\ 
P_{(d,0),\pm} &=& (d \pm \sqrt{d(d-1)}, \sqrt{d(d-1)}(\sqrt{d} \pm \sqrt{d-1})),
\end{eqnarray*}
along with their negatives (note that one can also prove Lemma~\ref{4torsion}
using these expressions). Applying Lemma~\ref{2descent} to these points gives:

\begin{lemma}\label{8torsion}
The following conditions are both necessary and sufficient for the
points $P_{(0,0),\pm}, P_{(1,0),\pm}$ and $P_{(d,0),\pm}$ respectively, to be halvable: 
\begin{itemize}
\item[(i)] $P_{(0,0),\pm} \in 2L_d(\F_q) \Longleftrightarrow
\pm\sqrt{d},\pm\sqrt{d}-1,\pm \sqrt{d} -d \in (\F_{q}^{\times})^2$,
\item[(ii)] $P_{(1,0),\pm} \in 2L_d(\F_q) \Longleftrightarrow
1 \pm \sqrt{1-d},\pm \sqrt{1-d},1 \pm \sqrt{1-d} -d \in (\F_{q}^{\times})^2$,
\item[(iii)] $P_{(d,0),\pm} \in 2L_d(\F_q) \Longleftrightarrow
d \pm \sqrt{d(d-1)},d \pm \sqrt{d(d-1)}-1,\pm \sqrt{d(d-1)} \in (\F_{q}^{\times})^2$.
\end{itemize}
\end{lemma}

\subsection{The case $q \equiv -1 \pmod{4}$}

This is the simplest case, giving rise to the following theorem:

\begin{theorem}\label{pm1main}
If $q \equiv -1 \pmod{4}$, then the following holds:
\begin{itemize}
\item[(i)] Let $a^4 \in \F_q \setminus \{0,1\}$. Then $\#L_{a^4}(\F_q) = p+1-A \equiv 0 \pmod{8}$.

\item[(ii)] Conversely, if $q+1-A \equiv 0 \pmod{8}$ then there exists $a^4 \in
  \F_q \setminus \{0,1\}$ such that $\#L_{a^4}(\F_q) = q+1-A$.

\item[(iii)] If $q+1 - A \equiv 0 \pmod{8}$ then $N_{4}(A) = N_{2}(A) = 2N(A)/3$.
\end{itemize}
\end{theorem}

\begin{proof}
Since $a^4$ is a square, by Table 2 we have 
$L_{a^4}(\F_q)[4] \cong \Z_4 \times \Z_2$, 
and hence by Lagrange's theorem we have $8 \mid \#L_{a^4}(\F_q)$. This proves $(i)$. 
Now let $A$ be any unobstructed integer satisfying $q+1 - A \equiv 0 \pmod{8}$,
and consider the set of all curves $L_d(\F_q)$ counted by $N(A)$. By Lemma~\ref{ordinary}
this set is non-empty. By Theorem~\ref{notsquare2} we have $N_{2}(A) =
2N(A)/3$. Furthermore, since $q \equiv -1 \bmod{4}$, the map $x^2 \mapsto x^4$ is an automorphism of the set
of squares in $\F_q \setminus \{0,1\}$, and hence 
$N_{4}(A) = N_{2}(A)$. This proves $(iii)$ and hence $(ii)$.
\end{proof}

\subsection{The case $q \equiv 1 \pmod{4}$}

We have the following theorem, which is proven in the remainder of this section:

\begin{theorem}\label{p14main}
If $q \equiv 1 \pmod{4}$, then the following holds:
\begin{itemize}
\item[(i)] Let $a^4 \in \F_q \setminus \{0,1\}$. Then $\#L_{a^4}(\F_q) = q+1-A \equiv 0 \pmod{16}$.

\item[(ii)] Conversely, if $q+1-A \equiv 0 \pmod{16}$ then there exists $a^4 \in
  \F_q \setminus \{0,1\}$ such that $\#L_{a^4}(\F_q) = q+1-A$.

\item[(iii)] If $q+1 - A \equiv 0 \pmod{16}$ then $N_{4}(A) = N(A) - 2N(-A)$.
\end{itemize}
\end{theorem}

Note that the implication in $(iii)$ is equivalent to $N_4(A)/N(A) = 1 - 2/r$, where $r$ is Katz's
ratio $N(A)/N(-A)$. Using Theorem~\ref{1mod4nonsquare} and~(\ref{partition}),
this is equivalent to $N_4(A) = N_{n2}(A) + N_{2n4}(A) + N_{4}(A) - 2N_{n2}(A)$, or 
\begin{equation}\label{mainbijection}
N_{2n4}(A) = N_{n2}(A). 
\end{equation}
Equation~(\ref{mainbijection}) in fact holds for all $A$ such that $q+1 - A \equiv
0 \pmod{8}$, and seems to be non-trivial. We will prove it by constructing a bijection between
the sets of curve parameters of each type. Once this equality is proven, part $(ii)$ follows easily.

The idea behind the proof of Equation~(\ref{mainbijection}) is a natural extension of the bijection-based proofs
of \S\ref{complete} , which used the isomorphisms given in Lemma~\ref{easyisos}. Rather than use
isomorphisms defined over $\F_q$, which are isogenies of degree one, we use 
isogenies of degree two. In particular we consider the isomorphism classes 
of curves arising from two $2$-isogenies of $L_d$: the first being ``divide by the $\Z/2\Z$ generated
by $(0,0)$'' when $d \in N_{2n4}(A)$, and the second being ``divide by the $\Z/2\Z$ generated
by $(1,0)$'' when $d \in N_{n2}(A)$, which are dual to one another. 
We begin with a short proof of part $(i)$.
\vspace{4mm}
\newline
\noindent 
{\em Proof of (i):} Let $d = a^4$. Since $\chi_2(d) = 1$, by Table 1, if $\chi_2(1-d)=1$ then
$L_d(\F_q)[4] \equiv \Z_4 \times \Z_4$ and hence $16 \mid \#L_d(\F_q)$. If
$\chi_2(1-d)=-1$ then by Table 1 neither of $(1,0)$ or $(d,0)$ are
halvable, and we claim that precisely one of $P_{(0,0),\pm}$ is
halvable. As $\chi_2(-1) = 1$, by Lemma~\ref{8torsion}, $P_{(0,0),+}$ is halvable if and only if $\sqrt{d},
\sqrt{d}-1$ are both square, while $P_{(0,0),-}$ is halvable if and only if $-\sqrt{d},
-\sqrt{d}-1$ are both square. Since $\chi_4(d) = 1$, both $\pm \sqrt{d}$ are
square. Furthermore, as $1-d = (1+\sqrt{d})(1-\sqrt{d}) =
(-\sqrt{d}-1)(\sqrt{d}-1)$, precisely one of these factors is square as
$\chi_2(1-d)=-1$ by assumption. This gives rise
to a point of order $8$. Therefore $\L_d(\F_q)[8] \cong \Z_8 \times \Z_2$ 
and hence $16 \mid \#L_d(\F_q)$ in this case too. This completes the proof of $(i)$.
\vspace{4mm}
\newline
We now exhibit a bijection to prove~(\ref{mainbijection}), assuming $q+1-A
\equiv 0 \pmod{8}$. 

\begin{lemma}
Let $A$ satisfy $q+1 - A \equiv 0 \pmod{8}$. Then there exists an injection
from $N_{2n4}(A)$ to $N_{n2}(A)$.
\end{lemma}

\begin{proof}
Note that if $d \in N_{2n4}(A)$ then by Table 1 we necessarily have $q+1-A \equiv 0
\pmod{8}$. Let $\xi_d: L_d \rightarrow L_d/\langle (0,0) \rangle$ and let $E^d = \xi_d(L_d)$.
Using V\'{e}lu's formula~\cite{velu}, $E^d$ has equation $y^2 = x^3 - (d+1)x^2 - 4dx + 4d(d+1) = (x - (d+1))(x -
2\sqrt{d})(x + 2\sqrt{d})$, and
\[
\xi_d(x,y) = ( x + d/x, y(1 - d/x^2)).
\]
In particular, $(1,0),(d,0) \in L_d$ are both mapped to $(d+1,0) \in E^d$, and
hence $L_d$ is isomorphic to $E^d/\langle (d+1,0) \rangle$.

Labelling the abscissae of the order $2$ points of $E^d$ by $e_1= d+1,e_2=2\sqrt{d}$ and
$e_3=-2\sqrt{d}$, one sees~(\cite{washington}) that $E^d$ has six isomorphic Legendre curves, each given by a permutation
of $(e_1,e_2,e_3)$ with paramater $\lambda = (e_3 - e_1)/(e_2 - e_1)$, and 
\[
(x,y) \mapsto \bigg( \frac{x-e_1}{e_2-e_1}, \frac{y}{(e_2-e_1)^{3/2}} \bigg).
\]
Each of these isomorphisms is defined over $\F_q$ if and only if $\lambda \in \F_q$ and
$\chi_2(e_2 - e_1) =1$~\cite{washington}. For $d \in N_{2n4}(A)$, the two $E^d$-isomorphic Legendre curves used in
the bijection are given in Table 3.

\begin{table}[here]
\caption{$L_d/\langle (0,0) \rangle$-isomorphic Legendre curves in $N_{n2}(A)$
  for $d \in N_{2n4}(A)$}\label{tab3}
\begin{center}
\begin{tabular}{c|c|c|c}
\hline
$(e_1,e_2,e_3)$     & $\lambda$ & $(e_2-e_1)$ & $\chi_2(e_2-e_1)$ \\
\hline
$(2\sqrt{d},d+1, -2\sqrt{d})$  & $\lambda_1 = \frac{-4\sqrt{d}}{(1-\sqrt{d})^2}$ 
& $(1-\sqrt{d})^2$ & $1$ \\
$(-2\sqrt{d},d+1,2\sqrt{d})$  & $\lambda_2 =\frac{4\sqrt{d}}{(1+\sqrt{d})^2}$
& $(1+\sqrt{d})^2$ & $1$ \\
\hline
\end{tabular}
\end{center}
\end{table}

Observe that $\lambda_1(d),\lambda_2(d) \in N_{n2}(A)$ since $\chi_4(d) \ne 1$.  
Note also that $\lambda_1 = 1 - \delta$, with $\delta$ as given in \S\ref{Eiso}, and hence 
this isomorphism class is precisely that of $E_d$; indeed we have $j_L(\delta) = j_E(d)$.
Thus $E^d \cong E_d$, explaining our choice of notation.

Abusing notation slightly, we refer to the isomorphisms 
$E^d \rightarrow L_{\lambda_1(d)}$ and $E^d \rightarrow  L_{\lambda_2(d)}$ by $\lambda_1(d)$ and $\lambda_2(d)$
respectively. Note that both $\lambda_1(d)$ and $\lambda_2(d)$ map $(d+1,0) \in E^d$ to
$(1,0) \in L_{\lambda_1(d)},L_{\lambda_2(d)}$. 
Furthermore, if $d$ is replaced with $1/d$ in Table 3, then each $\lambda_i(d)$ 
remains invariant. Hence $L_{1/d}$ maps to $\lambda_1(d),\lambda_2(d)$ as well, via
$\xi_{1/d}(L_{1/d}) = E^{1/d}$, and
the point $(1/d + 1,0) \in E^{1/d}$ maps to $(1,0) \in L_{\lambda_1(d)},L_{\lambda_2(d)}$. 
As $1/d \in N_{2n4}(A)$, this means we
have a map from the pair $\{d,1/d\} \subset N_{2n4}(A)$ to the pair
$\{\lambda_1(d),\lambda_2(d)\} \subset N_{n2}(A)$. 
Note that $d,1/d$ are distinct, unless $d=-1$ and $q \equiv 5 \pmod{8}$, in which case we have
$\lambda_1(-1) = \lambda_2(-1) = 2$ with $\chi_2(2) = -1$ and hence $2 \in
N_{n2}(A)$. So in this exceptional case, we have an injection.

In the general case we thus have two pairs of maps: 
\begin{eqnarray}
\nonumber \lambda_1(d) \circ \xi_d : L_d &\longrightarrow& L_{\lambda_1(d)},\\
\nonumber \lambda_2(d) \circ \xi_d : L_d &\longrightarrow& L_{\lambda_2(d)},\\
\nonumber \lambda_1(1/d) \circ \xi_{1/d} : L_{1/d} &\longrightarrow& L_{\lambda_1(1/d)},\\
\nonumber \lambda_2(1/d) \circ \xi_{1/d} : L_{1/d} &\longrightarrow& L_{\lambda_2(1/d)},
\end{eqnarray}
with $L_{\lambda_1(d)} = L_{\lambda_1(1/d)}$ and $L_{\lambda_2(d)} = L_{\lambda_2(1/d)}$. 
We claim the above four maps taken together form an injective map from pairs $\{d,1/d\}$ to
pairs $\{\lambda_1(d),\lambda_2(d)\}$. Indeed suppose that for $d' \in N_{2n4}(A)$ we have 
\[
\lambda_1(d') =
\lambda_1(d) \ \text{or} \ \lambda_2(d),
\ \text{or} \ 
\lambda_2(d') =
\lambda_1(d) \ \text{or} \ \lambda_2(d).
\]
Then $\sqrt{d'} = \pm \sqrt{d}$ or $\sqrt{d'} = \pm 1/\sqrt{d}$, \ie $d' = d$
or $d' = 1/d$, and hence the map is injective on the stated pairs. 
\end{proof}

Now consider the reverse direction, which is almost immediate. 

\begin{lemma}
Let $A$ satisfy $q+1 - A \equiv 0 \pmod{8}$. Then there exists an injection
from $N_{n2}(A)$ to $N_{2n4}(A)$.
\end{lemma}

\begin{proof}
Let $e \in N_{n2}(A)$. For $q+1-A \equiv 0 \pmod{8}$, by Table 1 we must have $\chi_2(e) = -1$ and
$\chi_2(1-e)=1$. The only isomorphism defined over $\F_q$ in this case 
maps $ L_e \longrightarrow L_{e/(e-1)}$ (see (15)). 
Therefore if $e \in N_{n2}(A)$, then $\frac{e}{e-1} \in N_{n2}(A)$.
Indeed $\lambda_2(d) = \lambda_1(d)/(\lambda_1(d)-1)$ (and $\lambda_1(d) = \lambda_2(d)/(\lambda_2(d)-1)$).

Since $\lambda_1(d)$ and $\lambda_2(d)$ map the $\hat{\xi}_d$-generating element $(d+1,0)$
of $E^d$ to $(1,0)$ in $L_{\lambda_1}$ and $L_{\lambda_2}$ (and similarly
$(1/d+1,0) \in E^{1/d}$ to $(1,0)$), the dual $\hat{\xi}_d$ of $\xi_d$ applied
to the isomorphism class representative $L_{e}$ is given by $L_e/\langle (1,0) \rangle$, and
similarly for $L_{e/(e-1)}$. Hence if $e = \lambda_1(d)$ or $\lambda_2(d)$, then $\hat{\xi}_d$ 
maps elements of $N_{n2}(A)$ to the original isomorphism class of $L_d$.
We now analyse this map and identify which curves in the resulting isomorphism class are relevant.

For the sake of generality let $\gamma_e : L_e \longrightarrow L_e/\langle (1,0)
\rangle$ and let $F^e = \gamma_e(L_e)$. Using V\'{e}lu's formula $F^e$ has equation 
$y^2 = x^3 - (e+1)x^2 - (6e-5)x - 4e^2 +7e -3 = 
(x - (e-1))(x -(1 + 2\sqrt{1-e}))(x -(1 - 2\sqrt{1-e}))$, and
\[
\gamma_e(x,y) := \bigg( x + \frac{1-e}{x-1}, y\bigg(1 - \frac{1-e}{(x-1)^2}\bigg) \bigg).
\]
Note that $\gamma_e(0,0) = \gamma_e(e,0) = (e-1,0)$. For $e \in N_{n2}(A)$, the two $F^e$-isomorphic Legendre curves used in
the bijection are given in Table 4.

\begin{table}[here]
\caption{Two $L_e/\langle (1,0) \rangle$-isomorphic Legendre curves in
  $N_{2n4}(A)$ for $e \in N_{n2}(A)$ and $q+1-A \equiv 0 \pmod{8}$}
\begin{center}
\begin{tabular}{c|c|c}
\hline
$(e_1,e_2,e_3)$     & $\mu$ & $(e_2-e_1)$ \\
\hline
$(e-1,1 + 2\sqrt{1-e},1-2\sqrt{1-e})$  & $\mu_1 = \big(\frac{1-\sqrt{1-e}}{1+\sqrt{1-e}} \big)^2$
& $(1+\sqrt{1-e})^2$  \\
$(e-1,1-2\sqrt{1-e},1+2\sqrt{1-e})$  & $\mu_2 = \big(\frac{1+\sqrt{1-e}}{1-\sqrt{1-e}} \big)^2$
& $(1 - \sqrt{1-e})^2$  \\
\hline
\end{tabular}
\end{center}
\end{table}

Observe that $\chi_2(e_2 - e_1) = 1$ in each case, and the same is true for $\mu_1(e),\mu_2(e)$. 
Furthermore, $\mu_1(e),\mu_2(e)$ are both in
$N_{2n4}(A)$ since $(1  \pm \sqrt{1-e})/(1 \mp \sqrt{1-e})$ is not square. 
Indeed, since $1-e$ is square, write $1-e = b^2$ so that $e = 1- b^2 = (1+b)(1-b)$.
Therefore $-1 = \chi_2(e) = \chi_2(1+b)\chi_2(1-b) = \chi_2(1+b)/\chi_2(1-b) = \chi_2((1+b)/(1-b))$.

Again abusing notation slightly, we refer to the isomorphisms 
$F^e \rightarrow L_{\mu_1(e)}$ and $F^e \rightarrow  L_{\mu_2(e)}$ by $\mu_1(e)$ and $\mu_2(e)$
respectively.
Note that both $\mu_1(e)$ and $\mu_2(e)$ map $(e-1,0) \in F^e$ to
$(0,0) \in L_{\mu_1(e)},L_{\mu_2(e)}$. 
Furthermore, if $e$ is replaced with $e/(e-1)$ in Table 4, then each $\mu_i(e)$ 
remains invariant. Hence $L_{e/(e-1)}$ maps to $\mu_1(e),\mu_2(e)$ as well, via
$\gamma_{e/(e-1)}(L_{e/(e-1)}) = F^{e/(e-1)}$, and
the point $(e/(e-1) - 1,0) \in F^{e/(e-1)}$ maps to $(0,0) \in L_{\mu_1(e)},L_{\mu_2(e)}$. 
As $e/(e-1) \in N_{n2}(A)$, this means we
have a map from the pair $\{e,e/(e-1)\} \subset N_{n2}(A)$ to the pair
$\{\mu_1(e),\mu_2(e)\} \subset N_{2n4}(A)$. 
Note that $e,e/(e-1)$ are distinct, unless $e = 2$ and $q \equiv 5 \pmod{8}$, in which case we have
$\mu_1(2) = \mu_2(2) = -1$ with $\chi_4(-1) \ne 1$ and hence $-1 \in
N_{2n4}(A)$. So in this exceptional case, we have an injection (in fact the inverse of the previous injection).

In the general case we thus have two pairs of maps: 
\begin{eqnarray}
\nonumber \mu_1(e) \circ \gamma_e : L_e &\longrightarrow& L_{\mu_1(e)},\\
\nonumber \mu_2(e) \circ \gamma_e : L_e &\longrightarrow& L_{\mu_2(e)},\\
\nonumber \mu_1(e/(e-1)) \circ \gamma_{e/(e-1)} : L_{e/(e-1)} &\longrightarrow& L_{\mu_1(e/(e-1))},\\
\nonumber \mu_2(e/(e-1)) \circ \gamma_{e/(e-1)} : L_{e/(e-1)} &\longrightarrow& L_{\mu_2(e/(e-1))},
\end{eqnarray}
with $L_{\mu_1(e)} = L_{\mu_1(e/(e-1))}$ and $L_{\mu_2(e)} = L_{\mu_2(e/(e-1))}$. 
We claim the above four maps taken together form an injective map from pairs $\{e,e/(e-1)\}$ to
pairs $\{\mu_1(e),\mu_2(e)\}$. Indeed suppose that for $e' \in N_{n2}(A)$ we have 
\[
\mu_1(e') =
\mu_1(e) \ \text{or} \ \mu_2(e),
\ \text{or} \ 
\mu_2(e') =
\mu_1(e) \ \text{or} \ \mu_2(e).
\]
Then $e' = e$ or $e' = e/(e-1)$, and hence the map is injective on the stated pairs. 
\end{proof}

We have thus proven:
\begin{theorem}\label{bijproof}
Let $A$ satisfy $q+1 - A \equiv 0 \pmod{8}$. Then there exists a bijection
between $N_{2n4}(A)$ and $N_{n2}(A)$.
\end{theorem}

Furthermore, using the above definitions one can check that
\[
\mu_1(\lambda_1(d)) = \mu_1(\lambda_2(d)) = d, \ \text{and} \ \mu_2(\lambda_1(d)) = \mu_2(\lambda_2(d)) = 1/d,
\]
and similarly
\[
\lambda_1(\mu_1(e)) = \lambda_1(\mu_2(e)) = e/(e-1), \ \text{and} \ \lambda_2(\mu_1(e)) = \lambda_2(\mu_2(e)) = e,
\]
and that
\begin{eqnarray}
\nonumber (\mu_2(\lambda_1(d)) \circ \gamma_{\lambda_1(d)}) \circ (\lambda_1(d) \circ \xi_d) &=& [2]
\ \text{on} \ L_d,\\
\nonumber (\mu_2(\lambda_2(d)) \circ \gamma_{\lambda_2(d)}) \circ (\lambda_2(d) \circ \xi_d) &=& [2]
\ \text{on} \ L_d,\\
\nonumber (\lambda_1(\mu_1(e)) \circ \xi_{\mu_1(e)}) \circ (\mu_1(e) \circ \gamma_{e})  &=& [2] 
\ \text{on} \ L_e,\\
\nonumber (\lambda_1(\mu_2(e)) \circ \xi_{\mu_2(e)}) \circ (\mu_2(e) \circ \gamma_{e})  &=& [2] 
\ \text{on} \ L_e.
\end{eqnarray}

Observe that if one substitutes $d \in N_{2n4}(A)$ for $e$ in the latter two maps, then one
obtains $2$-isogenies from $L_d$ to $L_{\mu_1(d)}$,$L_{\mu_2(d)}$, however
$\mu_1(d), \mu_2(d) \not\in N_{n2}(A)$, so the bijection can only be used in the manner proven.
So while the bijection principally relies on a $2$-isogeny and its dual, this
alone is insufficient; one needs to also consider the isomorphism class
representatives used, which is natural given that we are considering Legendre
curve parameters $d$ rather than isomorphism classes of curves.

With regard to Theorem~\ref{p14main}, note that Theorem~\ref{bijproof}
directly implies Theorem~\ref{p14main}(iii).
Let $A$ be any unobstructed integer satisfying $q+1 - A \equiv 0 \pmod{16}$, and consider the set of all 
curves $L_d(\F_q)$ counted by $N(A)$. By Lemma~\ref{ordinary} this set is non-empty.
Theorems~\ref{katz2.8} and~\ref{katz2.11} show that for $q+1-A \equiv 0 \pmod{16}$ the ratio $N(A)/N(-A)
> 2$ and thus $N_4(A) = N(A) - 2N(-A) > 0$, which thus proves part $(ii)$ and completes the proof.


\section{Curves defined using a ratio of two quadratics}\label{quad}

Following on from \S\ref{pointcount} where we expressed the equation defining
$E_d$ in the form~(\ref{frac}), in this section we briefly discuss curves defined using a ratio of two 
quadratic polynomials or a ratio of a quadratic and a linear polynomial. We
demonstrate that one can derive an addition formula for these types
of curves and prove for them results similar to the results of the preceeding sections. 

\subsection{Ratio of two quadratics} 
Let $f(x)=a_1x^2+b_1x+c_1, g(x)=a_2x^2+b_2x+c_2\in \fq[x]$ be as in Lemma~\ref{q-c-s}, 
$a_1,a_2$ both non-zero, and define a curve by the equation
\begin{equation}\label{curveC}
C/\F_q: y^2={\frac{a_1x^2+b_1x+c_1}{a_2x^2+b_2x+c_2}}.
\end{equation}
Notice that writing the curve equation as a ratio of two quadratics is just for the sake of the 
exposition and it is understood that $C/\fq$ is the projective curve defined by the equation
$$
(a_2x^2+b_2xz+c_2z^2)y^2=a_1x^2z^2+b_1xz^3+c_2z^4.
$$

Now suppose that 
$$
f(x)=a_1(x-\omega_1)(x-\omega_2),
$$
and 
$$
g(x)=a_2(x-\omega_3)(x-\omega_4).
$$
The conditions of Lemma~\ref{q-c-s} imply that $\omega_1,\omega_2,\omega_3,\omega_4$ are
pairwise distinct. This implies that there is a linear fractional transformation
$$
\phi: x \mapsto \frac{u_1x+u_2}{u_3x+u_4}\;\;\;\;u_i\in\overline{\fq},
$$
which maps $\omega_1,\omega_2,\omega_3,\omega_4$ to $\mu,-\mu,1/\mu,-1/\mu$ provided
that the cross-ratio condition
$$ 
\frac{(\omega_1-\omega_3)(\omega_2-\omega_4)}{(\omega_2-\omega_3)(\omega_1-\omega_4)}=
\left(\frac{\mu^2-1}{\mu^2+1}\right)^2
$$
 is satisfied (see~[Chapter 4]\cite{needham}). The map $\phi$ induces the map
$$
\psi: x \mapsto \frac{-u_4x+u_2}{u_3x-u_1}$$
which in turn induces an isomorphism of the function field $\overline{\fq}(C)$ and the function field
of the curve $E^\mu$, $\overline{\fq}(E^\mu)$, where $E^{\mu}$ is defined by:
$$
y^2=\frac{x^2-\mu^2}{x^2-\frac{1}{\mu^2}}.
$$

$E^\mu$ is clearly isomorphic to the original Edwards curve~(\ref{original}). Thus
$\overline{\fq}(C)$ is an elliptic function field and hence the desingularization
of $C$ yields an elliptic curve. One can obtain results similar to the ones 
proven in~\cite{edwards} for the curve $C$. For example, one can obtain an addition formula
for the points on $C$ by using the Edwards curve addition formula and the map $\phi$,
as $\phi$ induces a group homomorphism between the group of points on $C$ and
the group of points on $E^\mu$.

\subsection{Ratio of a quadratic and a linear polynomial}

Now suppose that for the curve~(\ref{curveC}) we have $a_2=0, b_2\neq 0$, giving the
corresponding curve 
\begin{equation}\label{quadlin}
C'/\F_q: y^2={\frac{a_1x^2+b_1x+c_1}{b_2x+c_2}}.
\end{equation}
Then there is a linear fractional transformation 
\[
\varphi: x \mapsto \frac{u_{1}^{'}x+u_{2}^{'}}{u_{3}^{'}x+u_{4}^{'}},\;\; u_{i}^{'}\in \overline{\F_q},
\]
which maps $C'$ to a curve of the form~\ref{curveC} defined by a ratio of two quadratics, and
which induces an isomorphism between the function fields of $C'$ and a curve
of the form~\ref{curveC}. 
Thus our discussion in the previous section applies to curves defined
using the ratio of a quadratic and a linear polynomial.

\subsubsection{Huff curves}
The Huff's model of elliptic curves introduced by Huff~\cite{huff}
which has recently captured the interest of the cryptographic community~\cite{joy,wu-feng}
can be transformed to one of the form~(\ref{quadlin}).
In particular, the Huff's curve, defined by the equation 
\[
H_{a,b}/\fq: ax(y^2-1)=by(x^2-1),
\]
is transformed to the curve
\[
y^2=\frac{bt^2+at}{at+b},
\] 
by setting $xy=t$. Thus one can generate the Huff's curve addition law using the process 
outlined in the previous section. Furthermore, whenever a curve family is
$\F_q$-isomorphic to an Edwards or Legendre curve, one can deduce some
properties of the isogeny classes. For example, we have~\cite{wu-feng} 
\[
H_{a,b} \cong E_{\big(\frac{a-b}{a+b}\big)^2} \ \text{over} \ \F_q,
\]
and so applying Theorems~\ref{1mod4nonsquare} and~\ref{notsquare2} we conclude that for any
unobstructed $A$, if $q+1 - A \equiv 0 \pmod{8}$ then there exists a Huff's curve over $\F_q$ with that cardinality.
One can also apply the results of this paper directly to the Jacobi intersection family~\cite{jacobi}
\[
x^2 + y^2 = 1 \ \text{and} \ dx^2 + z^2 = 1,
\]
since this family has $j$-invariant $j_L(d)$.

\begin{remark}
A new single-parameter family of elliptic curves was introduced
in~\cite{Fre-Cas} (amongst more than $50,000$ others) defined 
by the curve equation
\[
Ax+x^2-xy^2+1=0,
\] 
which enjoys a uniform $x$-coordinate addition formula.
The curve equation can be rewritten as
\[
y^2=\frac{x^2+Ax+1}{x}.
\]
Hence one can obtain addition formula for this family of curves using the addition law of
Edwards curves, although we do not claim that this method generates the most
efficient group law. 
\end{remark}


\section{Concluding remarks}

We have identified the set of isogeny classes of Edwards curves over finite fields of odd characteristic, 
and have found the proportion of parameters $d$ in each isogeny class which give rise to complete Edwards curves.
Furthermore, we have identified the set of isogeny classes of original Edwards curves, and proven similar proportion results
for this sub-family of curves.

Although not included in the paper, by analysing the $4$- and $8$-torsion of
Legendre curves, and using variants of the established bijections, we were able to prove parts of
Katz's ratio theorems. We believe an interesting and challenging problem is
whether or not the methods of this paper can be developed to provide an alternative proof for
all parts of Katz's ratio theorems; and conversely, can Katz's methods be used to find relationships between
$N_{2^k}(A)$ and $N(A)$ similar to those proven in 
Theorems~\ref{1mod4nonsquare},~\ref{notsquare2},~\ref{pm1main}(iii) and~\ref{p14main}(iii), for $k>2$ and $q \equiv 1 \pmod{2^k}$?

\section*{Acknowledgements}

The authors would like to thank Steven Galbraith for offering some useful initial
pointers, and for answering all our questions.

\bibliographystyle{plain}

\bibliography{EI}

\end{document}